\documentclass[amstex, dvopdfmx, 11pt]{amsart}
\usepackage{graphicx}
\usepackage{amssymb}
\usepackage{mathrsfs}
\newtheorem{thm}{Theorem}[section]
\newtheorem{Pro}{Proposition}[section]
\newtheorem{lemma}{Lemma}[section]
\newtheorem{Cor}{Corollary}[section]
\newtheorem{Mythm}{Theorem}

\theoremstyle{definition}
\newtheorem{Def}{Definition}[section]

\theoremstyle{definition}
\newtheorem{Rem}{Remark}[section]
\numberwithin{equation}{section}
\makeatletter
\@namedef{subjclassname@2020}{\textup{2020} Mathematics Subject Classification}
\makeatother

\begin{document}
\title[]{Structures of Moduli Spaces of \\Generalized Cantor Sets in $(0, 1)^{\mathbb N}$}
\author{
  Hiroshige Shiga    
}
\address{Emeritus Professor at Tokyo Institute of Technology (Science Tokyo) \\
and
Osaka Central Advanced Mathematical Institute \\
3-3-138\ Sugimoto, Sumiyoshi-ku \ Osaka 558-8585 \ Japan
} 
\email{shiga.hiroshige.i35@kyoto-u.jp}
\date{\today}    

\begin{abstract}
For each $\omega\in (0, 1)^{\mathbb N}$, we may construct a Cantor set $E(\omega)\subset [0, 1]$ called a generalized Cantor set for $\omega$.
We study the moduli space of $\omega$ denoted by $\mathcal M(\omega)\subset (0, 1)^{\mathbb N}$.
It is the set of $\omega'$ so that $E(\omega')$ is quasiconformally equivalent to $E(\omega)$.
In this paper, we show that the set $\mathcal M(\omega)$ is measurable in $(0, 1)^{\mathbb N}$ and we give a necessary condition for $\omega'$ to belong to $\mathcal M(\omega)$.
 By using this condition, we show that there are uncountably many moduli spaces in $(0, 1)^{\mathbb N}$.
We also show that except for at most one moduli space, the volume of the moduli space with respect to the standard product measure of $(0, 1)^{\mathbb N}$ vanishes.
\end{abstract}
\maketitle
\section{Introduction and  main results}
Let $\mathcal C$ be a Cantor set in the complex plane $\mathbb C$.
It is a totally disconnected perfect set in $\mathbb C$.
We consider the moduli space $\mathcal M_{\mathcal C}$ of $\mathcal C$, the set of Cantor sets $\mathcal C'$ which are quasiconformally equivalent to $\mathcal C$.
Here, we say that a Cantor set $\mathcal C'$ in $\mathbb C$ is quasiconformally equivalent to $\mathcal C$  if there exists a quasiconformal mapping $\varphi :{\mathbb C}\to {\mathbb C}$ such that $\varphi (\mathcal C')=\mathcal C$.

In a series of papers (\cite{ShigaSchottky}, \cite{Shigadynamical}, \cite{ShigaTohoku}) ,  we have studied the quasiconformal equivalence of Cantor sets and their moduli spaces.
For example, we show that the limit set of a Schottky group and Cantor Julia sets of hyperbolic rational functions belong to the moduli space of the standard middle one third Cantor set $\mathcal C_0$.
Moreover, we show that the limit set of a generalized Schottky group with a rank one parabolic does not belong to the moduli space of $\mathcal C_0$.
Recently, R. Hidalgo (\cite{Hidalgo})  shows that the number of  the moduli spaces of Cantor sets generated by finitely generated Kleinian groups is at most four.

In this paper, we consider generalized Cantor sets $E(\omega)$ which are defined by $\omega =(q_n)_{n=1}^{\infty}\in \Omega:= (0, 1)^{\mathbb N}$.
The construction of $E(\omega)$ is a generalization of that of $\mathcal C_0$ as follows.
To construct $E(\omega)$ for $\omega =(q_n)_{n=1}^{\infty}\in (0, 1)^{\mathbb N}$, we start with the closed unit interval $I=[0, 1]$.
At the first step, we remove an open interval $J_1^1$ of length $q_1(=q_1|I|)$ from $I$ so that $I\setminus J_1^1$ consists of two closed intervals $I_1^1$ and $I_1^2$ of the same length.
We put $E_1(\omega)=I_1^1\cup I_1^2$.
We repeat this operation inductively.
Namely, for $E_{n-1}(\omega)=\cup_{j=1}^{2^{n-1}}I_{n-1}^j$, we remove open intervals $J_n^{2j-1}$ of the length $q_n |I_{n-1}^{j}|$ from $I_{n-1}^j$ so that $I_{n-1}^{j}\setminus J_{n}^{2j-1}$ consists of two closed intervals $I_{n}^{2j-1}$ and $I_{n}^{2j}$ of the same length $(j=1, \dots , 2^{n-1})$.
Thus, we obtain $E_{n}(\omega)=\cup_{j=1}^{2^n}I_{n}^{j}$ and $I\setminus E_n(\omega)=\cup_{j=1}^{2^n-1}J_n^{j}$, where $J_n^{2j}=J_{n-1}^{j}$.
Then, $E(\omega):=\cap_{n=1}^{\infty}E_n(\omega)$ is a Cantor set and we call it the generalized Cantor set for $\omega=(q_n)_{n=1}^{\infty}\in \Omega$.
Therefore, the set of generalized Cantor sets is identified with $\Omega$.

For a given $\omega\in \Omega$, we consider  the moduli space of $\omega$ denoted by $\mathcal M(\omega)$.
It is the set of $\omega'\in \Omega$ such that $E(\omega')$ is quasiconformally equivalent to $E(\omega)$.
In this paper, we study the structure of $\mathcal M(\omega)$ in $\Omega$.
First, we show the following:
\begin{Mythm}

\label{Thm:mesurable}
	Under the standard product topology of $\Omega$ $(=(0, 1)^{\mathbb N})$, the moduli space $\mathcal M(\omega)$ is an $F_{\sigma}$ set. In particular, it is a Borel measurable set in $\Omega$.
\end{Mythm}

The space $\Omega$ has the standard probability measure $m_{\infty}$ defined by the Lebesgue measure on $(0, 1)$.
Since each moduli space $\mathcal M(\omega)$ for $\omega\in \Omega$ is measurable, we may consider the volume $m_{\infty}(\mathcal M(\omega))$ of $\mathcal M(\omega)$.

\begin{Mythm}
\label{Thm:Volume}
	Except for at most one moduli space, the volume of any moduli space with respect to $m_{\infty}$ is zero.
\end{Mythm}

In \cite{ShigaTohoku}, we have conjectured that $m_{\infty}(\mathcal M(\omega))=0$ for any $\omega\in\Omega$.
While we have given a partial affirmative answer in Theorem \ref{Thm:Volume}, the conjecture is still open.

We consider conditions for $\omega'\in \Omega$ to belong to $\mathcal M(\omega)$.
\begin{Mythm}
\label{Thm:condition}
	Let $\omega=(q_n)_{n=1}^{\infty}$ and $\omega'=(q_n')_{n=1}^{\infty}$ be in $\Omega$.
	Suppose that both $\omega$ and $\omega'$  are increasing.
	Then $\omega'\in \mathcal M(\omega)$ only if
	\begin{equation}
	\label{eqn:Log}
		\sup_{n\in \mathbb N}\left |\log \frac{\log (1-q_n)}{\log (1-q_n')}\right |<\infty.
	\end{equation}
	\end{Mythm}

By using this result, we may show the following:
\begin{Mythm}
\label{Thm:uncountable}
	There are uncountably many moduli spaces of generalized Cantor sets in $\Omega$.
\end{Mythm}

The result may be interesting if it is compared with the fact  by R. Hidalgo mentioned above, namely, the finiteness of the number of Moduli spaces of Cantor sets generated by finitely generated Kleinian groups.

\medskip

{\bf Acknowledgment.} The author thanks the anonymous referee for helpful comments and suggestions which improve this paper.

\section{Preliminaries}
\subsection{Infinite sequences in $(0, 1)$ and generalized Cantor sets }
As we have mentioned in \S 1, the set of generalized Cantor sets is identified with $\Omega:=(0, 1)^{\mathbb N}$.

\begin{Def}
	Let $\omega=(q_n)_{n=1}^{\infty}\in \Omega:=(0, 1)^{\mathbb N}$.
We say that $\omega$ has \emph{a lower bound} and \emph{an upper bound} if $q_n\geq \delta$ and $q_n\leq\delta$ $(n\in \mathbb N)$  for some $\delta\in (0, 1)$ independent of $n$, respectively.
\end{Def}

We have already obtained a sufficient condition for two elements with a lower bound to be quasiconformally equivalent to each other.
\begin{thm}[cf. \cite{Shigadynamical} Theorem II,  Corollary 1.2 and Theorem III]
\label{thm:old}
	Let $\omega=(q_n)_{n=1}^{\infty}, \omega'=(q_n')_{n=1}^{\infty}$ be in $\Omega$ with a lower bound.
	Then $E(\omega')$ is quasiconformally equivalent to $E(\omega)$ if $\sup_{n\in \mathbb N}\left |\log \frac{1-q_n}{1-q_n'}\right |<\infty$.
	In particular, if $\omega$ has lower and upper bounds, then the Cantor set $E(\omega)$ is quasiconformally equivalent to the standard middle one third Cantor set $\mathcal C_0$.
	
	Moreover, if $\omega$ does not have an upper bound, then $E(\omega)$ is not quasiconformally equivalent to $\mathcal C_0$.
\end{thm}

\begin{Def}
	We say that $\omega=(q_n)_{n=1}^{\infty}\in\Omega$ is \emph{increasing} if $q_{n}\leq q_{n+1}<1$ for any $n\in \mathbb N$.
\end{Def}

For $\omega=(q_n)_{n=1}^{\infty}\in \Omega$, we take $E_k(\omega)=\cup_{j=1}^{2^k}I_k^j$ which is the closed set at the $k$-step of the construction given in the introduction.
Here, $I_k^j$ is a closed interval located at the right of $I_k^{j-1}$ $(j=1, 2, \dots , 2^k-1)$.
The set $[0, 1]\setminus E_k(\omega)$ consists of $2^k-1$ open intervals $J_k^1, \dots , J_k^{2^k-1}$.
The interval $J_k^j$ is located between $I_k^j$ and $I_k^{j+1}$.

Now, we calculate the lengths of $I_k^j$ and $J_k^j$ for the argument in \S 4.

Because of the construction of $E_k(\omega)$, we have
\begin{equation*}
	|I_k^j|=\frac{1}{2}(1-q_k)|I_{k-1}^j|=2^{-k}\prod_{i=1}^k (1-q_i)\quad (j\in \{1, \dots , 2^k\}).
\end{equation*}

Viewing the construction of $E_k(\omega)$ from $E_{k-1}(\omega)$, we see that the open interval $I_{k-1}^j$ is divided into $I_k^{2j-1}, J_k^{2j-1}$ and $I_k^{2j}$.
Hence, if $j$ is odd, then
\begin{equation}
\label{eqn:odd}
	|J_k^j|=q_k|I_{k-1}^j|=\frac{2q_k}{1-q_k}|I_k^j|.
\end{equation}

When $j$ is even, we put $j=2^{\ell}m$ for some integer $\ell$ with $1\leq \ell\leq k-1$ and some odd number $m$.
Since $J_k^j$ is located between $I_k^j$ and $I_k^{j+1}$, we see that $J_k^j=J_{k-1}^{j/2}=J_{k-1}^{2^{\ell -1}}$.
Repeating this argument, we have $J_{k}^j=J_{k-\ell}^m$. 
Since $m$ is odd, we have
\begin{eqnarray}
	|J_k^j|&=&|J_{k-\ell}^m|=\frac{2q_{k-\ell}}{1-q_{k-\ell}}|I_{k-\ell}^m| 
	=\frac{2q_{k-\ell}}{1-q_{k-\ell}}\cdot2^{-k+\ell}\prod_{i=1}^{k-\ell}(1-q_i) \nonumber \\
	\label{eqn:even}
	&=& \frac{2^{\ell+1}q_{k-\ell}}{\prod_{i=k-\ell}^k(1-q_i)}|I_k^j|.
\end{eqnarray}

\subsection{Topology of $\Omega$ and sequences of the complements of Cantor sets.}
As we have mentioned in \S 2, the set of generalized Cantor sets is identified with $\Omega:=(0, 1)^{\mathbb N}$.
We consider the space $\Omega$ with the weak topology (cf.~\cite{Bog}).
That is, a sequence $\omega_m=(q_n^m)_{n=1}^{\infty}$ $(m=1, 2, \dots )$ converges to $\omega=(q_n)_{n=1}^{\infty}$ if and only if $\lim_{m\to \infty}q_n^m=q_n$ for each $n\in \mathbb N$.

For a sequence $\{\omega_m\}$ converging to $\omega$, we consider the complements $X(\omega_m)$ of $E(\omega_m)$ $(m\in \mathbb N)$ and the  \emph{Carath\'eodory kernel} $N(\{X(\omega_m)\}_{m=1}^{\infty})$ of the sequence of the domains.
Here, $z\in \widehat{\mathbb C}$ belongs to the kernel $N(\{X(\omega_m\}_{m=1}^{\infty})$ if and only if there exist a neighborhood $U$ of $z$ and $M\in \mathbb N$ such that $U\subset X(\omega_m)$ for all $m>M$ (cf. \cite{Lehto-Virtanen}).

\begin{Pro}
\label{Pro:kernel}
	Suppose that $\{\omega_m\}_{m=1}^{\infty}\subset \Omega$ converges to $\omega$ in $\Omega$ as $m\to \infty$.
	Then, $N(\{X(\omega_m)\}_{m=1}^{\infty})=X(\omega)$.
\end{Pro}
\begin{proof}
	If $z\not\in \mathbb R$ or $z=\infty$, then obviously, $z\in N(\{X(\omega_m)\}_{m=1}^{\infty}$.
	Also, it is obvious that $z\in N(\{X(\omega_m)\}_{m=1}^{\infty})$ if $z\in (-\infty, 0)\cup (1, +\infty)$.
	
	Let $z$ be in $[0, 1]\setminus E(\omega)$.
	Then, there exists $k\in \mathbb N$ such that $z\in [0, 1]\setminus E_k(\omega)$.
	Since $E_k(\omega)=\cup_{j=1}^{2^k}I_k^j$, the point $z$ belongs to $I_k^J$ for some $J\in \{1, 2, \dots , 2^k\}$.
		Take a sufficiently small neighborhood $U$ of $z$ so that $U\cap \mathbb R\subset I_k^J$.	
	Let $I_{k, m}^J$ be the interval of $E_k(\omega_m)$ corresponding to $I_k^J$.
	The intervals $I_k^J$ and $I_{k, m}^J$ depend only on $(q_n)_{n=1}^k$ and $(q_n^m)_{n=1}^k$, and $q_n^m\to q_n$ for $n=1, \dots k$.
	Hence, $I_{k, m}^J$ is close to $I_k^J$ if $m$ is sufficiently large.
	 Therefore, we see that $U\cap \mathbb R\subset I_{k, m}^J\subset X(\omega_m)$ for a sufficiently large $m$ and we are convinced that $X(\omega)\subset N(\{X(\omega_m)\}_{m=1}^{\infty})$.
	 
	 For $x\in E(\omega)$, take a neighborhood $U$ of $z$.
	 From the construction of $E(\omega)$, there exist $k\in \mathbb N$, $J\in \{1, 2, \dots , 2^k\}$ and an interval $I_k^J$ of $E_k(\omega)$ such that $U\supset I_k^J$.
	 Because of the same reason as above, we may find an interval $I_{k, m}^J$ for a large $M\in \mathbb N$ such that $U\supset I_{k, M}^J$.
	 This implies that $U\not\subset X(\omega_M)$.
	 
	 For any $m>M$, there exists an interval $I_{k, m}^j$ such that $I_{k, m}^j\subset I_{k, M}^J$.
	 Thus, we see that $U\supset I_{k, m}^j$ and
	  $U\not\subset X(\omega_m)$ for $m>M$.
	  So, we conclude that $X(\omega)=N(\{X(\omega_m)\}_{m=1}^{\infty})$ and the proof is completed.
	 \end{proof}
	 
	 For a sequence $\varphi_n : G\to G_n$ of $K$-quasiconformal mappings $(n=1, 2, \dots)$, the Carath\'eodory kernel $N(\{G_n\}_{n=1}^{\infty})$ becomes the image of the limit mapping if the limit mapping is also a quasiconformal mapping.
	 \begin{thm}[\cite{Lehto-Virtanen} Chapter II, Theorem 5.4 ]
	 \label{thm:kernel}
Let $G\subset\widehat{\mathbb C}$ be a domain with at least two boundary points.
	Suppose that $K$-quasiconformal mappings $\varphi_n : G\to G_n$ converges to a quasiconformal mapping $\varphi$ locally uniformly on $G$ as $n\to \infty$.
	Then $\varphi$ maps $G$ onto a component of $N(\{G_n\}_{n=1}^{\infty})$.
\end{thm}

\subsection{The canonical pants decomposition (cf.~\cite{Shigadynamical} \S 3)}
We put $X(\omega)=\widehat{\mathbb C}\setminus E(\omega)$ for $\omega=(q_n)_{n=1}^{\infty}\in \Omega:=(0, 1)^{\mathbb N}$.
The surface $X(\omega)$ is a hyperbolic Riemann surface of infinite type.
We define a pants decomposition of $X(\omega)$ associated to $\omega$.

First, we put $\gamma_0:=\{ix\mid x\in \mathbb R\}\cup\{\infty\}$.
It is a simple closed geodesic in $X(\omega)$.

Next, we consider $E_2(\omega)=\cup_{j=1}^4 I_2^j$ and take circles $C_2^j$  passing through two open intervals adjacent to $I_2^j$ $(j=1, \dots , 4)$.
Then, the geodesics $\gamma_2^j$ in $X(\omega)$ homotopic to $C_2^j$ together with $\gamma_0$ make two pairs of pants, $P_1^1$ and $P_1^2$ so that $\partial P_1^1=\gamma_0\cup\gamma_2^1\cup\gamma_2^2$ and $\partial P_1^2=\gamma_0\cup\gamma_2^3\cup\gamma_2^4$.

Inductively, we construct a pair of pants $P_k^j$ for $k\in \mathbb N$ and $j\in \{1, 2, \dots 2^k\}$.
Each $P_k^j$ is bounded by geodesics $\gamma_k^j, \gamma_{k+1}^{2j-1}$ and $\gamma_{k+1}^{2j}$, and $\{P_k^j\}_{k\in \mathbb N, j=1, 2, \dots , 2^k}$ gives a pants decomposition of $X(\omega)$.
We call it \emph{the canonical pants decomposition} of $X(\omega)$.

\subsection{Extremal length, the modulus of a ring domain and holomorphic quadratic differentials.}

Let $W$ be a domain in $\widehat {\mathbb C}$ (or a Riemann surface) and $\Gamma$ a family of curves in $W$.
We consider a measurable conformal metric $\rho=\rho (z)|dz|$ on $W$ with
\begin{equation*}
	A(\rho):=\iint_W \rho^2(z)dxdy\in (0, +\infty).
\end{equation*}
Then, the extremal length $\lambda (\Gamma)$ of $\Gamma$ is defined by
\begin{equation*}
	\lambda (\Gamma)=\sup_{\rho}\frac{L_{\Gamma}(\rho)^2}{A(\rho)},
\end{equation*}
where $L_{\Gamma}(\rho)=\inf_{\gamma\in \Gamma} \int_{\gamma}\rho (z)|dz|$ and the supremum is taken over all above conformal metrics $\rho$.
Obviously, the extremal length is a conformal invariant.

A domain $W$ in $\widehat{\mathbb C}$ is called \emph{a ring domain} if $\partial W$ consists of two connected components.
The modulus of a ring domain $W$ denoted by $\textrm{Mod}(W)$ is defined by
\begin{equation*}
	\textrm{Mod}(W)=\lambda (\Gamma_W)^{-1},
\end{equation*}
where $\Gamma_W$ is the set of closed curves in $W$ separating the two connected components of $\partial W$.
It is known that both $\lambda(\Gamma_W)$ and $\lambda(\Gamma_Q)$ above have extremal metrics which attain the supremums of the definition of the extremal length (cf. \cite{Ahlfors}).

More precisely, for a circular annulus $A(1, R)=\{z\in \mathbb C \mid 1<|z|<R\}$ $(1<R)$ and a quadrilateral $(A(1, R)\cap\mathbb H ; -R, -1, 1, R)$, metrics $\rho_0 :=\frac{1}{|z|}|dz|\arrowvert_{A(1, R)}$ and $\rho_0\arrowvert_{A(1, R)\cap\mathbb H}$ is the extremal metric, respectively.
This fact is generalized as follows (cf. \cite{K.Strebel}) .

Let $R$ be a hyperbolic Riemann surface of finite type and $\varphi=\varphi (z)dz^2$ be a holomorphic quadratic differential on $R$ with $\|\varphi\|:=\iint |\varphi|<\infty$.
Then, a branch of $w=\int_{p_0}^{p}\sqrt{\varphi}$ gives a local coordinate around $p_0\in R$ unless $p_0$ is a zero of $\varphi$.
A \emph{horizontal trajectory} of $\varphi$ is the locus of $\{\textrm{Re } w=c\}$ for some $c\in \mathbb R$.
We say that a horizontal trajectory of $\varphi$ is \emph{closed} if it is a simple closed curve on $R$.
Then, the following is known:
\begin{Pro}
\label{Pro:Jenkins-Strebel}
	Let $\gamma$ be a simple closed hyperbolic geodesic on $R$.
	Then, there exists a holomorphic quadratic differential $\varphi_{\gamma}$ on $R$ with $\|\varphi_{\gamma}\|<\infty$ such that
	\begin{enumerate}
		\item any closed horizontal trajectory of $\varphi_{\gamma}$ is homotopic to $\gamma$,
		\item the set of closed horizontal trajectories of $\varphi_{\gamma}$ forms an annulus $A_{\gamma}$ and
	\begin{equation*}
		\textrm{Mod}(A_{\gamma})=\lambda (\Gamma_{\gamma})^{-1},
	\end{equation*}
	where $\Gamma_{\gamma}$ is the set of simple closed curves homotopic to $\gamma$ on $R$.
	\end{enumerate}
	\end{Pro}

Finally, we exhibit a theorem by Maskit on the comparison of the hyperbolic length and the extremal length.
\begin{thm}[Maskit \cite{Maskit}]
\label{eqn:Maskit}
	Let $\ell_{R}(\gamma)$ be the hyperbolic length of a closed geodesic $\gamma$ in a hyperbolic Riemann surface $R$ and $\lambda (\Gamma)$ the extremal length of the curve family $\Gamma$ of closed curve homotopic to $\gamma$ in $R$.
	Then,
	\begin{equation}
		\lambda (\Gamma)\leq \frac{1}{2}\ell_{R}(\gamma)e^{\ell_{R}(\gamma)/2}
	\end{equation}
\end{thm} 
holds.

\section{Proof of Theorem \ref{Thm:mesurable}}
For $\omega\in \Omega$, we consider a function $F_{\omega} : \Omega\to \mathbb R_{\geq 1}\cup\{+\infty\}$ by
\begin{equation*}
	F_{\omega}(\omega')=\begin{cases}
		\inf_{\varphi} K(\varphi), &\omega'\in \mathcal M(\omega) \\
		+\infty, &\omega'\not\in \mathcal M(\omega),
	\end{cases}
\end{equation*}
where the infimum is taken over all quasiconformal mappings $\varphi :\mathbb C\to \mathbb C$ with $\varphi (E(\omega))=E(\omega')$ and $K(\varphi)$ is the maximal dilatation of $\varphi$.

\begin{lemma}
	$F_{\omega}$ is a lower semi-continuous function on $\Omega$.
\end{lemma}
\begin{proof}
Take $\omega_0\in \Omega$ and we show that $F_{\omega}$ is lower semi-continuous at $\omega_0$.

At first, we assume that $\omega_0\in \mathcal M(\omega)$.
		For $\omega_0=(q_n^0)_{n=1}^{\infty}\in \mathcal M(\omega)$, we consider $\omega_m=(q_n^m)_{n=1}^{\infty}$ converging to $\omega_0$ as $m\to \infty$.
		
		Since $F_{\omega}(\omega_0)<+\infty$, we may assume that $K:=\liminf_{m\to\infty}F_{\omega}(\omega_m)<+\infty$.
		Therefore, for any $\varepsilon>0$, there exists $N\in \mathbb N$ such that
		$K(\varphi_m)\leq K+\varepsilon$ for any $m\in \mathbb N$ with $m>N$, where $\varphi_m : \mathbb C\to \mathbb C$ $(m=1, 2, \dots )$ are extremal quasiconformal mappings with $\varphi (E(\omega_0))=E(\omega_m)$.  Hence,  $K(\varphi_m)=F_{\omega}(\omega_m)$.
		Therefore, $\{\varphi_m\}_{m=1}^{\infty}$ is a normal family.
		
		Let $\varphi$ be a limit map of a subsequence of the family.
		There are the following three possibilities of $\varphi$ (\cite{Lehto-Virtanen} Chapter II). 
		\begin{enumerate}
			\item $\varphi\equiv \infty$;
			\item $\varphi\equiv z_0$ for some $z_0\in \mathbb C$;
			\item $\varphi$ is a quasiconformal self map of $\mathbb C$.
		\end{enumerate}
		
		Since the convergence to $\varphi$ is locally uniform on $\mathbb C$ and $\varphi_m (E(\omega))=E(\omega_m)$, we see that only the case (3) is possible.
		Thus, $K(\varphi)\leq K+\varepsilon$ and we obtain
		\begin{equation*}
			K(\varphi)\leq K=\liminf_{m\to\infty}F_{\omega}(\omega_m).
		\end{equation*}

		On the other hand, since $\omega_m\to \omega_0$, we have seen in Proposition \ref{Pro:kernel} that 
		the Carath\'eodory kernel $N(\{E(\omega_m)^c\}_{m=1}^{\infty})$ of the sequence domains $\{E(\omega_m)^c\}_{m=1}^{\infty}$ is $E(\omega_0)^c$.
		It follows from Theorem \ref{thm:kernel} that $\lim_{m\to \infty}\varphi_m (E(\omega)^c)=\varphi (E(\omega)^c)=E(\omega_0)^c$ and $\varphi (E(\omega))=E(\omega_0)$.
		Therefore, we obtain
		\begin{equation*}
			F_{\omega}(\omega_0)\leq K(\varphi)\leq  K=\liminf_{m\to\infty}F_{\omega}(\omega_m).
		\end{equation*}
		
		If $\omega_0\not\in \mathcal M(\omega)$, then $F_{\omega}(\omega_0)=+\infty$.
		Take a sequence $\{\omega_m\}_{m=1}^{\infty}$ in $(0, 1)^{\mathbb N}$ which converges to $\omega_0$.
		
		If $\omega_m\not\in \mathcal M(\omega)$ for any $m\in \mathbb N$, then $F_{\omega}(\omega_m)=+\infty$ and we have
		\begin{equation*}
			F_{\omega}(\omega_0)= +\infty=\lim_{m\to\infty}F_{\omega}(\omega_m).
		\end{equation*}
		
		If the sequence contains infinitely many elements in $\mathcal M(\omega)$, say $ \omega_{m_k}$ $(k=1, 2, \dots)$, then we obtain
		\begin{equation*}
			\lim_{k\to \infty}F_{\omega}(\omega_{m_k})=+\infty=F_{\omega}(\omega_0).
		\end{equation*}
		Indeed, if $\liminf_{k\to\infty}F_{\omega}(\omega_{m_k})<+\infty$, then we conclude that $F_{\omega}(\omega_0)<+\infty$ by using the normal family argument as in the first part of the proof above and it is a contradiction.
		
		Thus, we verify that $F_{\omega}$ is lower semi-continuous at $\omega_0$ as desired.
	\end{proof}
	Now, the proof of Theorem \ref{Thm:mesurable} is obvious.
	Indeed, from the definition of $F_{\omega}$, we have
	\begin{equation*}
		\mathcal M(\omega)=F_{\omega}^{-1}([1, +\infty))=\cup_{n=1}^{\infty}F_{\omega}^{-1}([1, n]).
	\end{equation*} 
	Since $F_{\omega}$ is lower semi-continuous, $F_{\omega}^{-1}([1, n])$ $(n=1, 2, \dots )$ are closed sets in $\Omega$.
	Hence, $\mathcal M(\omega)$ is an $F_{\sigma}$ set.
	This completes the proof of Theorem \ref{Thm:mesurable}.
	
	\section{Proof of Theorem \ref{Thm:Volume}}
	\subsection{The shift map and quasiconformal equivalence}
	Let $\sigma :\Omega\to \Omega$ be the shift map.
	Namely, for $\omega=(q_n)_{n=1}^{\infty}\in \Omega$,
	\begin{equation*}
		\sigma (\omega)=(q_{n+1})_{n=1}^{\infty}=(q_2, q_3, \dots ).
	\end{equation*}

	The shift map $\sigma$ is a classical object in the ergodic theory.
	Actually, the followings are well-known (cf. \cite{Ergodic}).
	\begin{Pro}
		\begin{enumerate}
			\item The shift map $\sigma$ is measure preserving, that is, $m_{\infty}(\sigma^{-1}(A))=m_{\infty}(A)$ for any measurable set $A\subset \Omega$.
			\item The action of $\sigma : \Omega\to \Omega$ is mixing with respect to $m_{\infty}$.
Namely, for any measurable sets $A, B$ in $\Omega$, 
\begin{equation}
\label{eqn:mixing}
	\lim_{n\to\infty}m_{\infty}(A\cap \sigma^{-n}(B))=m_{\infty}(A)m_{\infty}(B)
\end{equation}
holds.
As mixing is a stronger property than ergodicity, the action of $\sigma$ on $\Omega$ is ergodic.
		\end{enumerate}
	\end{Pro}
	
	Now, we establish a relationship between the shift map and the quasiconformal equivalence.
	
	\begin{lemma}
	\label{lemma:shift and qc}
	Let $\omega_i =(q_n^{(i)})_{n=1}^{\infty}$ be in $\Omega$ $(i=1, 2)$.
		If $\sigma (\omega_1)$ and $\sigma (\omega_2)$ are quasiconformally equivalent to each other, then so are $\omega_1$ and $\omega_2$. 
		In particular, if $\omega, \omega'\in \sigma^{-1}(\alpha)$ for some $\alpha\in \Omega$, then $\mathcal M(\omega)=\mathcal M(\omega')$.
	\end{lemma}
	\begin{proof}
		Suppose that $\sigma (\omega_2)$ is quasiconformally equivalent to $\sigma (\omega_1)$.
		Then, there exists a quasiconformal mapping $\varphi : \mathbb C\to \mathbb C$ with $\varphi (E(\sigma (\omega_1)))=E(\sigma (\omega_2))$.
		
		We take $W(\sigma (\omega_i)):=X(\sigma (\omega_i))\cap \Delta (2^{-1}, \{(1-q_1^{(i)})q_2^{(i)}\}^{-1}-2^{-1})$ ($i=1, 2$), where $X(\omega )=\widehat{\mathbb C}\setminus E(\omega)$  for $\omega\in \Omega$ and $\Delta (a, r)$ is an open disk of radius $r>0$ centered at $a\in \mathbb C$.
		From the construction, $W(\sigma (\omega_i))$ is conformally equivalent to $W_i:=X(\omega_i)\cap \Delta (4^{-1}(1-q_1^{(i)})q_2^{(i)}, 2^{-1}(1-2^{-1}(1-q_1^{(i)})q_2^{(i)}))$ and $W_i' :=X(\omega_i)\cap \Delta (1-4^{-1}(1-q_1^{(i)})q_2^{(i)}, 2^{-1}(1-2^{-1}(1-q_1^{(i)})q_2^{(i)}))$.
		Let $F_i$ and $F_i'$ be a conformal mapping from $W(\sigma (\omega_i))$ to $W_i$ and $W_i'$, respectively $(i=1, 2)$.
		
		Since $\varphi$ is a homeomorphism with $\varphi (E(\sigma (\omega_1)))=E(\sigma (\omega_2))$, there exist Jordan domains $U_1, U_2$ such that $E(\sigma (\omega_1))\subset U_1\subset W(\sigma (\omega_1))$, $E(\sigma (\omega_2))\subset U_2\subset W(\sigma (\omega_2))$ and $U_2=\varphi (U_1)$.
		Then, $\psi_1:=F_2\circ \varphi \circ {F_1|_{F_1(U_1)}}^{-1} : F_1(U_1)\to F_2(U_2)$ and $\psi_1':=F_2'\circ \varphi \circ {F_1'|_{F_1'(U_1)}}^{-1} : F_1'(U_1)\to F_2'(U_2)$ give a quasiconformal mapping $\Phi$ from $F_1(U_1)\cup F_1'(U_1)$ to $F_2(U_2)\cup F_2'(U_2)$ with $\Phi (E(\omega_1))=E(\omega_2)$.
		We have also a quasiconformal mapping $\Psi : \widehat {\mathbb C}\setminus \overline{F_1(U_1)\cup F_1'(U_1)}\to \widehat {\mathbb C}\setminus\overline{F_2(U_2)\cup F_2'(U_2)}$  with $\Psi (\partial F_1(U_1))=\partial F_2(U_2)$ and $\Psi (\partial F_1'(U_1))=\partial F_2'(U_2)$ since both regions are annuli. 
		
		Here, we use the following fact on the gluing of quasiconformal mappings ( \cite{ShigaSchottky} Lemma 4.1) .
		
		\begin{Pro}
 	\label{Pro:Gluing}
	Let $X, Y$ be Riemann surfaces. 
	We consider simple closed curves $\alpha\subset X$ and $\beta\subset Y$ with $X\setminus{\alpha}=X_{1}\sqcup X_2$ and $Y\setminus\beta =Y_1\sqcup Y_2$, respectively.
	Suppose that there exist quasiconformal mappings $f_{i} : X_{i}\to Y_{i}$ $(i=1, 2)$ such that $f_1(\alpha)=f_2(\alpha)=\beta$.
	Then, there exist an annular neighborhood $A$ of $\alpha$ and a quasiconformal mapping $f$ on $A$ into $Y$ such that
	\begin{enumerate}
		\item $A':=f(A)$ is an annular neighborhood of $\beta$;
		\item we put
	\begin{equation}
	\label{extendqc}
		F(p)=\begin{cases}
			f_i (p), &p\in X_i\setminus A \quad (i=1, 2) \\
			f(p), &p\in U.
		\end{cases}
	\end{equation}
		Then, $F$ is a quasiconformal mapping from $X$ to $Y$.
	\end{enumerate}
 \end{Pro}

By applying this Proposition for $f_1=\Phi, f_2=\Psi$, we verify that $\omega_1$ and $\omega_2$ are quasiconformally equivalent to each other.
		\end{proof}
		
\noindent
{\bf Notation.} We define $\mathcal M(\sigma^{-1}(\omega))$ by $\mathcal M(\omega')$ for some $\omega'\in \sigma^{-1}(\omega)$.
From the above lemma, $\mathcal M(\sigma^{-1}(\omega))$ depends on $\omega$ but it is independent of $\omega'\in \sigma^{-1}(\omega)$.

\begin{Cor}
	For $\omega=(q_k)_{k=1}^{\infty}\in \Omega$ and for any $n\in \mathbb N$,
	\begin{equation}
	\label{eqn:sigma-inclusion}
		\sigma^n (\mathcal M(\omega))\supset \mathcal M(\sigma^n (\omega)),\quad \mathcal M(\sigma^{-n}(\omega))\supset \sigma^{-n}(\mathcal M(\omega )).
	\end{equation}
	and
	\begin{equation}
	\label{eqn:measure=}
		\sigma^{-n}(\sigma^n(\mathcal M(\omega)))=\mathcal M(\omega).
	\end{equation}
\end{Cor}
\begin{proof}
	Let $\widetilde \omega$ be in $\mathcal M(\sigma (\omega))$.
	For $\omega'\in \sigma^{-1}(\widetilde \omega)$, $\widetilde \omega=\sigma (\omega')$ is quasiconformally equivalent to $\sigma (\omega)$.
	Hence, $\omega'$ is quasiconformally equivalent to $\omega$ by Lemma \ref{lemma:shift and qc}.
	This implies that $\widetilde \omega=\sigma (\omega ')$ belongs to $\sigma (\mathcal M(\omega))$ and we have $\sigma (\mathcal M(\omega))\supset \mathcal M(\sigma (\omega))$.
	Repeating this argument, we conclude that $\sigma^n (\mathcal M(\omega))\supset \mathcal M(\sigma^n (\omega))$.
	
	The same argument gives another inclusion of (\ref{eqn:sigma-inclusion}).
	Indeed, if $\widehat\omega=(\widehat{q}_k)_{k=1}^{\infty}$ belongs to $\sigma^{-n}(\mathcal M(\omega))$, then there exist $\omega'=(q_k')_{k=1}^{\infty}\in \mathcal M(\omega)$ and $(x_1, \dots , x_n)\in (0, 1)^n$ such that
	\begin{equation*}
		\widehat \omega=(x_1, \dots , x_n, q_1', q_2', \dots ).
	\end{equation*}
	Hence, $\sigma^n (\widehat \omega)=\omega'$ is quasiconformally equivalent to $\omega$ and we conclude that $\widehat\omega\in \mathcal M(\sigma^{-n}(\omega))$ from Lemma \ref{lemma:shift and qc}.
	
As for (\ref{eqn:measure=}), it suffices to show that $\sigma^{-n}(\sigma^n(\mathcal M(\omega)))\subset \mathcal M(\omega)$ because $\sigma^{-n}(\sigma^n(\mathcal M(\omega)))\supset \mathcal M(\omega)$ is obvious. 
	
	Let $\widehat\omega$ be in $\sigma^{-n}(\sigma^n(\mathcal M(\omega)))$.
	Then, there exist $\omega'=(q_k')_{k=1}^{\infty}\in \mathcal M(\omega)$ and $(x_1, \dots , x_n)\in (0, 1)^{n}$ such that
	\begin{equation*}
		\widehat \omega=(x_1, \dots , x_n, q_{n+1}', q_{n+2}', \dots ).
	\end{equation*}
	Hence, we have $\sigma^n(\widehat\omega)=\sigma^n(\omega')$.
	Thus, we conclude that $\widehat\omega$ is quasiconformally equivalent to $\omega'\in \mathcal M(\omega)$.
	So, $\widehat\omega$ belongs to $\mathcal M(\omega)$ and we have $\sigma^{-n}(\sigma^n (\mathcal M(\omega)))\subset\mathcal M(\omega)$.
	
\end{proof}		
\subsection{Proof of the theorem}
Since $\sigma$ is measure preserving, 
\begin{equation*}
	m_{\infty}(\sigma^{-n}(\mathcal M(\omega)))=m_{\infty}(\mathcal M(\omega)),
\end{equation*}
for any $n\in \mathbb N$.

Suppose that $m_{\infty}(\mathcal M(\omega))>0$. Then, from the mixing property (\ref{eqn:mixing}) we see that there exists $N\in \mathbb N$ such that
\begin{equation*}
	m_{\infty}(\sigma^{-N}(\mathcal M(\omega))\cap \mathcal M(\omega))>0
\end{equation*}

Since $\mathcal M(\sigma^{-N}(\omega))\supset\sigma^{-N}(\mathcal M(\omega))$ from (\ref{eqn:sigma-inclusion}), we have
\begin{equation*}
	m_{\infty}(\mathcal M(\sigma^{-N}(\omega))\cap\mathcal M(\omega))>0.
\end{equation*}
Thus, we conclude $\mathcal M(\sigma^{-N}(\omega))=\mathcal M(\omega)$ because moduli spaces are equivalence classes.
This means $\mathcal M(\omega)=\mathcal M(\sigma^{-N}(\omega))\supset \sigma^{-N}(\mathcal M(\omega))$.
Hence, we see that $\mathcal M(\omega)=\sigma^{-N}(\mathcal M(\omega))$ as measurable sets.
Since $\sigma^N$ is also mixing, it follows from (\ref{eqn:mixing}) that $m_{\infty}(\mathcal M(\omega))=1$ if $m_{\infty}(\mathcal M(\omega))>0$.
Thus, we conclude that at most one moduli space has the positive volume, if any.
	
	\section{Proof of Theorem \ref{Thm:condition}}
	
	We may assume that $\sup_{n\in \mathbb N} q_n =1$ and $\sup_{n\in \mathbb N} q_n '=1$.
	Indeed, if $\sup_{n\in \mathbb N} q_n <1$ and $E(\omega')$ is quasiconformally equivalent to $E(\omega)$, then it follows from Theorem \ref{thm:old} that $\sup_{n\in \mathbb N} q_n' <1$. 
	Indeed, since $\omega$ has both upper and lower bounds, $E(\omega)$ is quasiconformally equivalent to the standard middle one third Cantor set $\mathcal C_0$.
	However, if $\omega'$ does not have an upper bound, then $E(\omega')$ is not quasiconformally equivalent to $\mathcal C_0$ nor $E(\omega)$.
	This is a contradiction.
	Therefore, the condition (2) is satisfied in this case.
	The same argument also works when $\sup_{n\in \mathbb N} q_n' =1$
	So, we assume that $\sup_{n\in \mathbb N} q_n =1$ and $\sup_{n\in \mathbb N} q_n' =1$.
	
	Suppose that $E(\omega')$ is quasiconformally equivalent to $E(\omega)$ but the condition (\ref{eqn:Log}) is not satisfied, namely
	\begin{equation}
	\label{eqn:Log+}
		\sup_{n\in \mathbb N}  \left |\log \frac{\log (1-q_n)}{\log (1-q_n')}\right |=\infty.
	\end{equation}
	
	\noindent
	{\bf Step 1 : Analyzing the canonical pants decomposition (Part I).}
	
	\noindent	
	Let $\{P_k^j\}_{j=1, 2, \dots , 2^k ; k\in \mathbb N}$ (resp. $\{{P'}_k^{j}\}_{j=1, 2, \dots , 2^k ; k\in \mathbb N}$) be the canonical pants decomposition of $X(\omega)$ (resp. $X(\omega')$).
	Each $P_k^j$ (resp. ${P'}_k^i$) is bounded by geodesics $\gamma_k^j, \gamma_{k+1}^{2j-1} \mbox{ and }\gamma_{k+1}^{2j}$ (resp. $\gamma_k^j, \gamma_{k+1}^{2j-1} \mbox{ and }\gamma_{k+1}^{2j}$).
	Then, Kinjo (\cite{Kinjo}) shows the following:
	\begin{lemma}[Kinjo]
		\begin{enumerate}
			\item If $j=1$ or $2^k$, then
			\begin{equation*}
				\ell _{X(\omega)}(\gamma_k^j)\leq \frac{2\pi^2}{\log \frac{1+q_k}{1-q_k}};
			\end{equation*}
			\item If $j=2^{\ell}m$ or $2^{\ell}m+1$ for some $\ell\in \{1, \dots , k-1\}$ and an odd number $m$, then
			\begin{equation*}
				\ell _{X(\omega)}(\gamma_k^j)\leq \max \left \{\frac{2\pi^2}{\log \frac{1+q_k}{1-q_k}}, \frac{2\pi^2}{\log \frac{1-q_k+2q_{k-\ell}}{1-q_k}}\right \},
			\end{equation*}
		\end{enumerate}
		where $\ell _{X(\omega)}(c)$ is the hyperbolic length of a curve $c$ in $X(\omega)$.
	\end{lemma}

	Since both $\omega =(q_n)_{n=1}^{\infty}$ and $\omega'=(q_n')_{n=1}^{\infty}$ are increasing, we may assume that $q_n, q_n'>\delta$ $(n=1, 2, \dots)$ for some $\delta\in (0, 1)$.
	Hence, we obtain
	\begin{equation*}
		\frac{1+q_k}{1-q_k}, \frac{1-q_k+2q_{k-\ell}}{1-q_k}\geq 1+\frac{2\delta}{1-q_k}
	\end{equation*}
	and
	\begin{equation}
	\label{eqn:upperbound}
		\ell _{X(\omega)}(\gamma_k^j)\leq\frac{2\pi^2}{\log\left (1+ \frac{2\delta}{1-q_k}\right )}\leq\frac{2\pi^2}{\log\left (1+ \frac{2\delta}{1-\delta}\right )}.
	\end{equation}
	
	Similarly, we have
	\begin{equation}
		\ell _{X(\omega')}({\gamma'}_k^j)\leq\frac{2\pi^2}{\log\left (1+ \frac{2\delta}{1-q_k'}\right )}\leq\frac{2\pi^2}{\log\left (1+ \frac{2\delta}{1-\delta}\right )}.
	\end{equation}
	In particular, they are bounded above and converge to zero as $n\to \infty$ because $\lim_{n\to \infty}q_n=\lim_{n\to \infty}q'_n=1$.
	From the hyperbolic trigonometry, we see that short geodesics do not pass through any  pair of pants in the canonical pants decomposition.
	Actually, from the collar lemma (cf.~\cite{Buser} Chapter 4), we may obtain the following lemma immediately.
	
\begin{lemma}
	There exists a constant $d=d(\delta)$ depending only on $\delta$ such that the hyperbolic length of any geodesic curve $c$ in $P_k^j$ or ${P'}_k^j$ $(k\in \mathbb N; j\in \{1, 2, \dots , 2^k\})$ connecting two points on the boundary curves is not less than $d$.
\end{lemma}

	\noindent
	{\bf Step 2 : Analyzing the canonical pants decomposition (Part II : Extremal length).}
	
	\noindent
We estimate $\ell_{X(\omega)}(\gamma_k^j)$ and $\ell_{X(\omega')}(\gamma_{k}^j)$ from below.

	 When $j$ is even, $I_k^j$ is located between $J_k^{j-1}$ and $J_k^j$.
	As we have seen in (\ref{eqn:odd}) and (\ref{eqn:even}),
	\begin{equation*}
		|J_k^{j-1}|=\frac{2q_k}{1-q_k}|I_k^j|\leq \frac{2}{1-q_k}|I_k^j|
	\end{equation*}
	and
	\begin{equation*}
		|J_k^j|=\frac{2^{\ell+1}q_{k-\ell}}{\prod_{i=k-\ell}^k(1-q_i)}|I_k^j|\leq 2\left (\frac{2}{1-q_k}\right )^{k}|I_k^j|,
	\end{equation*}
	since $(q_n)_{n=1}^{\infty}$ is increasing.
		 We put $J_k^{j-1}=(x_1,  x_2), I_k^j=[x_2, x_3]$ and $J_k^j=(x_3, x_4)$ $(x_1<x_2<x_3<x_4)$.
	 
	 We consider a four-punctured sphere $X_{k, j}:=\widehat{\mathbb C}\setminus \cup_{i=1}^4 x_i$ and we denote by $\Gamma_0$ the set of closed curves on $X_{k, j}$ which are homotopic to $\gamma_{k}^j$ in $X_{k, j}$.
	 Now, we consider a holomorphic quadratic differential $\varphi_0$ on $X_{k, j}$ defined by
	 \begin{equation*}
	 	\varphi_0:=\frac{dz^2}{(z-x_1)(z-x_2)(z-x_3)(z-x_4)}.
	 \end{equation*}
	 We see that any closed horizontal trajectory of $\varphi_0$  belongs to $\Gamma_0$.
	 We also see that the set of closed horizontal trajectories of $\varphi_0$ forms  a ring domain $A_{k, j}:=\widehat{\mathbb C}\setminus ((-\infty, x_1]\cup [x_4, +\infty)\cup\{\infty\}\cup [x_2, x_3])$.
	 Hence, from Proposition \ref{Pro:Jenkins-Strebel}, we have
	 \begin{equation}
	 	\lambda (\Gamma_0)=(\textrm{Mod}(A_{k, j}))^{-1}.
	 \end{equation}

	  It is seen that the ring domain $A_{k, j}$ is conformally equivalent to a ring domain $B_{k, j}:=\widehat{\mathbb C}\setminus ([-\infty, -\alpha]\cup [0, 1]\cup [\beta, +\infty])$, where $\alpha=|J_{k}^{j-1}|/|I_{k}^j|\leq 2/(1-q_k)$ and $\beta=1+|J_k^j|/|I_k. j|\leq 1+2(2/(1-q_k))^k$.
	  Therefore,
	  \begin{equation*}
	  	\textrm{Mod}(A_{k, j})=\textrm{Mod}(B_{k, j}).
	  \end{equation*}
	 
	 For $\alpha'=2/(1-q_k)$ and $\beta' = 1+2(2/(1-q_k))^k=1+2\alpha'^k$, the ring domain $A_{k, j}':=\widehat{\mathbb C}\setminus ([-\infty, -\alpha']\cup [0, 1]\cup [\beta', +\infty])$ contains $B_{j, k}$ and we obtain
	 
	 \begin{equation*}
	 	\textrm{Mod}(B_{k, j})\leq \textrm{Mod}(A_{k, j}').
	 \end{equation*}
	 Furthermore, by taking a M\"obius transformation, we see that the domain $A_{k, j}'$ is conformally equivalent to a Teichm\"uller domain $A_{j, k}'':=\mathbb C\setminus ([-1, 0]\cup [\alpha' (\beta '-1)/(\alpha'+\beta'), +\infty)$ and
	 \begin{equation*}
	 	\textrm{Mod}(A_{k, j}')=\textrm{Mod}(A_{k, j}'').
	 \end{equation*}
	 
	 Now, we introduce a function $\Psi (t)$ which represents the modulus of the Teichm\"uller domain $T(t):=\mathbb C\setminus ([-1, 0]\cup [t, +\infty)$ for $t>0$.
	 Actually, $\textrm{Mod}(T(t))$, the modulus of $T(t)$ is $\log \Psi (t)$.
	 As for $\Psi (t)$, the following is known:
	 \begin{Pro}[\cite{Hariri-Klen-Vuorinen}]
	 	For $t>0$, $\Psi (t)$ is a strictly increasing function and
	 	\begin{equation*}
	 	t+1\leq	(\sqrt{t+1}+\sqrt{t})^2\leq \Psi (t)\leq 4(\sqrt{t+1}+\sqrt{t})^2\leq 16(t+1).
	 	\end{equation*}
	 \end{Pro}
	 Therefore, we have
	 \begin{equation*}
	 	\textrm{Mod}(A_{k, j}'')=\textrm{Mod}(T(\alpha' (\beta' -1)/(\alpha' +\beta')))\leq \log \left \{16\left (\frac{\alpha' (\beta' -1)}{\alpha'+\beta'}+1\right )\right \}.
	 \end{equation*}
	 	 In fact, since $\beta'=2 \alpha'^k+1$, we have
	 \begin{eqnarray*}
	 	\frac{\alpha'(\beta'-1)}{\alpha'+\beta'}+1&=&\frac{2\alpha'^{k+1}+2\alpha'^k+\alpha'+1}{2\alpha'^k+\alpha'+1} \\
	 	&=&\alpha'+1-\frac{\alpha'^2+\alpha}{2\alpha'^k+\alpha'+1}\leq \alpha'+1.
	 \end{eqnarray*}
	 and
	 \begin{equation*}
	 	\log \left \{16\left (\frac{\alpha' (\beta' -1)}{\alpha'+\beta'}+1\right )\right \}\leq 4\log 2+\log (\alpha'+1).
	 \end{equation*}
	 Hence, we have
	 \begin{equation*}
	 	\textrm{Mod}(A_{k, j}'')\leq 4\log 2+\log \left (\frac{2}{1-q_k}+1\right )\leq 6\log2-\log (1-q_k).
	 \end{equation*}
	 Thus, we obtain
	 \begin{equation}
	 	\textrm{Mod}(A_{k, j})\leq \textrm{Mod}(A_{k, j}'')\leq 6\log 2-\log (1-q_k)
	 \end{equation}
	 and
	 \begin{equation*}
	 	\lambda (\Gamma_0)=\left  (\textrm{Mod}(A_{k, j})\right )^{-1}\geq\frac{1}{6\log 2-\log (1-q_k)}.
	 \end{equation*}
	 It follows from (\ref{eqn:Maskit}) that
	 \begin{equation*}
	 	2\lambda (\Gamma_0)\leq \ell_{X_{k, j}}(\gamma_0)\exp\{\ell_{X_{k, j}}(\gamma_0)/2\},
	 \end{equation*}
	 where $\gamma_0$ is the hyperbolic geodesic in $\Gamma_0$.
	 Since $X_{k, j}\supset X(\omega)$, we obtain
	 \begin{equation*}
	 	\ell_{X_{k, j}}(\gamma_0)\leq\ell_{X(\omega)}(\gamma_{k}^j)
	 \end{equation*}
	 and
	 \begin{equation*}
	 	2\lambda (\Gamma_0)\leq \ell_{X_{k, j}}(\gamma_0)\exp\{\ell_{X_{k, j}}(\gamma_0)/2\}\leq \ell_{X(\omega)}(\gamma_{k}^j)\exp\{\ell_{X(\omega)}(\gamma_{k}^j)/2\}.
	 \end{equation*}
	 
	 On the other hand, we have seen in (\ref{eqn:upperbound}) that
	 \begin{equation*}
	 	\frac{1}{2}\ell_{X(\omega)}(\gamma_k^j)\leq\frac{\pi^2}{\log\left (1+\frac{2\delta}{1-\delta}\right )}=:C(\delta).
	 \end{equation*}
	 Therefore, we have
	 \begin{equation*}
	 	\ell_{X(\omega)}(\gamma_k^j)\geq 2e^{-C(\delta)}\lambda (\Gamma_0)\geq \frac{2e^{-C(\delta)}}{6\log 2-\log (1-q_k)}.
	 \end{equation*}
	 
	 Similarly, we have
	 \begin{equation*}
	 	\ell_{X(\omega')}({\gamma'}_k^j)\geq \frac{2e^{-C(\delta)}}{6\log 2-\log (1-q'_k)}.
.
	 \end{equation*}
	 
	 When $j$ is odd, we see that the above argument still works and we obtain the same estimates as above.
	 
	 Thus, we get the following estimates.
	 \begin{equation}
	 \label{eqn:n}
	 		 \frac{2e^{-C(\delta)}}{6\log 2-\log (1-q_k)}\leq
\ell _{X(\omega)}(\gamma_k^j)\leq\frac{2\pi^2}{\log\left (1+ \frac{2\delta}{1-q_k}\right )},
	 \end{equation}
	 and
	\begin{equation}
	\label{eqn:n'}
		 \frac{2e^{-C(\delta)}}{6\log 2-\log (1-q'_k)}\leq
\ell _{X(\omega)}({\gamma'}_k^j)\leq\frac{2\pi^2}{\log\left (1+ \frac{2\delta}{1-q_k'}\right )}
	\end{equation} 
	for each $j\in \{1, 2, \dots , 2^k\}$.
	
	We compare the left-hand side and the right-hand side of each other.
	\begin{eqnarray*}
		R(q_k, q_k')&:=&\frac{\pi ^2e^{C(\delta)}(6\log 2-\log (1-q'_k))}{\log \left (1+\frac{2\delta}{1-q_k}\right )} \\
		&=&\pi ^2 e^{C(\delta)}\frac{\log (1-q_k')}{\log (1-q_k)}\cdot \frac{\frac{6\log 2}{\log (1-q_k')}-1}{\frac{\log (2\delta+1-q_k)}{\log (1-q_k)}-1}.
	\end{eqnarray*}
	and
	\begin{eqnarray*}
		R(q_k', q_k)&:=&\frac{\pi ^2 e^{C(\delta)}(6\log 2-\log (1-q_k))}{\log \left (1+\frac{2\delta}{1-q_k'}\right )} \\
		&=&\pi ^2e^{C(\delta)}\frac{\log (1-q_k)}{\log (1-q_k')}\cdot \frac{\frac{6\log 2}{\log (1-q_k')}-1}{\frac{\log (2\delta+1-q_k')}{\log (1-q_k')}-1}.
	\end{eqnarray*}

Thus, we have the following :
\begin{lemma}
\label{lemma:Compare}
	If $q_k, q_k'\geq \delta>0$, then
	$\inf _{k\in \mathbb N}R(q_k, q_k')>0$ if and only if
	$$
	\inf _{k\in \mathbb N}\frac{\log (1-q_k')}{\log (1-q_k)}>0
	$$
	and $\inf_{k\in \mathbb N}R(q_k', q_k)> 0$ if and only if
	$$
	\inf  _{k\in \mathbb N}\frac{\log (1-q_k)}{\log (1-q_k')}>0.
	$$

\end{lemma}	

\begin{Rem}
	E. Kinjo (\cite{Kinjo}) has estimated $\ell_{X(\omega)}(\gamma_k^j)$ from below.
	Actually, she shows
	\begin{equation*}
		\ell_{X(\omega)}(\gamma_k^j)\geq 2\sinh^{-1}\left (1/\sinh \left (\frac{\pi^2}{\log \left (\frac{1+q_k}{2q_k}\right )}\right )\right ).
	\end{equation*}
	She has used a different method from ours, while the above estimate is weaker than ours.
\end{Rem}

\noindent
{\bf Step 3 : Looking at the images of geodesics by quasiconformal maps.}	

\noindent
Since we are supposing  (\ref{eqn:Log+}), there exists a subsequence $\{n_k\}_{k=1}^{\infty}$ such that $\lim_{k\to \infty}q_{n_k}=\lim_{k\to \infty}q'_{n_k}=1$ and
\begin{equation}
\label{eqn:limit}
		\lim_{k\to\infty}  \left |\log \frac{\log (1-q_{n_k})}{\log (1-q_{n_k'})}\right |=\infty.
\end{equation}
In particular, $\lim_{k\to \infty}\ell ({\gamma'}_{n_k}^j)=0$ by (\ref{eqn:n'}).

Suppose that there exists a $K$-quasiconformal mapping $\varphi : \mathbb C\to \mathbb C$ with $\varphi (E(\omega'))=E(\omega)$.
By Wolpert's lemma (\cite{Wolpert}), we have
\begin{equation*}
	\ell_{X(\omega)}([\varphi ({\gamma}'_{n_k})])\leq K\ell_{X(\omega')}({\gamma'}_{n_k}^j),
\end{equation*}
where $[\varphi (\gamma_{n_k}^j)]$ is the geodesic in $X(\omega)$ homotopic to $\varphi ({\gamma'}_{n_k}^j)$. 

Hence, $\lim_{k\to \infty}\ell_{X(\omega)}([\varphi ([\gamma_{n_k}^j])=0$ and $\ell_{X(\omega)}([\varphi ([\gamma_{n_k}^j])<d$ if $k$ is sufficiently large, where $d>0$ is a constant given in Lemma 4.2.
Therefore, the geodesic $[\varphi ({\gamma'}_{n_k}^j)]$ cannot pass through any pair of pants of the pants decomposition of $X(\omega)$.
Thus, we obtain:
\begin{lemma}
\label{lemma:Its boundary}
	If $k$ is sufficiently large, then $[\varphi ({\gamma'}_{n_k}^j)]$ is a boundary curve of a pair of pants of the canonical pants decomposition of $X(\omega)$.
\end{lemma}

\noindent
{\bf Step 4 : Comparing the lengths of geodesics.}

\noindent
From (\ref{eqn:limit}), the following two cases occurs.

\begin{equation*}
\mbox{Case I : }\quad
	\frac{\log (1-q_{n_k})}{\log (1-q'_{n_k})}\to 0 \quad 
	(k\to \infty)
	\end{equation*}
	or
\begin{equation*}
\mbox{Case II : }\quad
	\frac{\log (1-q'_{n_k})}{\log (1-q_{n_k})}\to 0 \quad (k\to \infty).
	\end{equation*}	 

\noindent
{\bf Case I.}
From Lemma \ref{lemma:Compare}, $R(q_{n_k}', q_{n_k})< K^{-1}$ for a sufficiently large $k$.
This implies that
\begin{eqnarray*}
	\ell_{X(\omega)}([\varphi ({\gamma'}_{n_k}^j)])&\leq& K\ell_{X(\omega')}({\gamma'}_{n_k}^j)\leq\frac{2\pi^2 K}{\log\left (1+ \frac{2\delta}{1-q_{n_k}'}\right )} \\
	&=& KR(q'_{n_k}, q_{n_k})\cdot\frac{2e^{C(\delta)}}{6\log 2-\log (1-q_k)}\\
	&\leq& K R(q'_{n_k}, q_{n_k})\ell_{X(\omega)}(\gamma_{n_k}^m)<\ell_{X(\omega)}(\gamma_{n_k}^m)
\end{eqnarray*}
for any $j, m\in \{1, 2, \dots , 2^{n_k}\}$.

Moreover, since $(q_n)_{n\in \mathbb N}$ is increasing, for any $n<n_k$, $q_n\leq q_{n_k}$ and
\begin{equation*}
	0<-\log \left (1-q_n\right )\leq-\log \left (1-q_{n_k}\right )
\end{equation*}
holds.
Therefore, we also obtain
\begin{eqnarray*}
	\ell_{X(\omega)}([\varphi ({\gamma'}_{n_k}^j)])	&\leq & KR(q'_{n_k}, q_{n_k})\cdot\frac{2e^{C(\delta)}}{6\log 2-\log \left (1-q_{n_k}\right )}\\
	&\leq & KR(q'_{n_k}, q_{n_k})\cdot\frac{2e^{C(\delta)}}{6\log 2-\log \left (1-q_n \right )}<\ell_{X(\omega)}(\gamma_{n}^m)
\end{eqnarray*}
for $n<n_k$ and $m\in\{1, 2, \dots , 2^n\}$.

Therefore, we see that the geodesic $[\varphi ({\gamma'}_{n_k}^j)]$, which is a boundary of a pair of pants in the canonical pants decomposition for $\omega$, has to be in the inside of a topological disk bounded by $\gamma_{n_k}^i$ of some $i\in \{1, 2, \dots , 2^{n_k}\}$.

\noindent
{\bf Case II.} From Lemma \ref{lemma:Compare}, $R(q_{n_k}, q_{n_k}')< K^{-1}$ for a sufficiently large $k$.
By using the same argument as in Case I above, we obtain
\begin{equation*}
	\ell_{X(\omega)}([\varphi({\gamma'}_{n_k}^j)])>\ell_{X(\omega)}(\gamma_{n_k}^m)
\end{equation*}
for any $j, m\in \{1, 2, \dots , 2^{n_k}\}$, and
\begin{equation*}
	\ell_{X(\omega)}([\varphi({\gamma'}_{n_k}^j)])>\ell_{X(\omega)}(\gamma_{n_k}^m)
\end{equation*}
for $n>n_k$ and $m\in\{1, 2, \dots , 2^n\}$.

Therefore, we see that the geodesic $[\varphi ({\gamma'}_{n_k}^j)]$,  has to be in the outside of any topological disk bounded by $\gamma_{n_k}^i$ $(i\in \{1, 2, \dots , 2^{n_k}\})$.

\medskip

\noindent
{\bf Step 5 : Using a topological argument to get a contradiction.}
Let $D_{n_k, j}$ (resp.~ $D'_{n_k, j}$) be the disk bounded by $\gamma_{n_k}^j$ (resp.~${\gamma'}_{n_k}^j$).
We put $W_{n_k, j}:=\widehat{\mathbb C}\setminus\cup_{j=1}^{2^{n_k}}\overline{D_{n_k,j}}$ and $W'_{n_k, j}:=\widehat{\mathbb C}\setminus\cup_{j=1}^{2^{n_k}}\overline{D'_{n_k, j}}$.
Both $W_{n_k, j}$ and $W'_{n_k, j}$ are planar domains bounded by  mutually disjoint $2^{n_k}$ simple closed curves and
\begin{equation*}
	W_{n_k, j}\cap E(\omega)=\emptyset, \quad W'_{n_k, j}\cap E(\omega')=\emptyset.
\end{equation*}

Since $\varphi : \mathbb C\to \mathbb C$ is a homeomorphism and $\varphi (E(\omega'))=E(\omega)$, the image $\varphi (W'_{n_k, j})$ is also a domain bounded by $2^{n_k}$ simple closed curves $\varphi ({\gamma'}_{n_k}^j)$ $(j=1, 2, \dots , 2^{n_k})$ and $\varphi (W'_{n_k, j})\cap E(\omega)=\emptyset$.

Let $\widehat{W}_{n_k, j}$ be a domain bounded by $[\varphi ({\gamma'}_{n_k}^j)]$ $(j=1, 2, \dots , 2^{n_k})$.
Since $[\varphi ({\gamma'}_{n_k}^j)]$ is homotopic to $\varphi ({\gamma'}_{n_k}^j)$, $\widehat{W}_{n_k, j}$ has the same property as that of $\varphi (W'_{n_k, j})$.
Namely, $\widehat{W}_{n_k, j}\cap E(\omega)=\emptyset$.
Moreover, $[\varphi ({\gamma'}_{n_k}^j)]$ is a boundary of a pair of pants of the canonical pants decomposition of $X(\omega)$ for a sufficiently large $k$ (Lemma \ref{lemma:Its boundary}).
Now, let us consider Case I in Step 3.

As we have seen there, each geodesic $[\varphi ({\gamma'}_k^j)]$ is in the inside of some $D_{n_k, i}$ $(1\leq i, j\leq 2^{n_k})$.
Furthermore, $[\varphi ({\gamma'}_k^j)]\not=\partial D_{n_k, i}=\gamma_{n_k}^i$ because $\ell_{X(\omega)}([\varphi ({\gamma'}_{n_k}^j)])<\ell_{X(\omega)}(\gamma_{n_k}^i)$.
Hence, the simply connected domain $D_{n_k, i}$ contains at least two geodesics in $\partial\widehat{W}_{n_k, j}$, otherwise $\widehat{W}_{n_k, j}\cap E(\omega)\not=\emptyset$.
This is absurd since both $W_{n_k, j}$ and $\widehat{W}_{n_k, j}$ are bounded by the same number of simple closed curves and they do not include any point in $E(\omega)$.

We may use the same argument for Case II and have a contradiction as well.
Thus, we have completed the proof of Theorem \ref{Thm:condition}.

\section{Proof of Theorem \ref{Thm:uncountable}}
The proof of Theorem \ref{Thm:uncountable} is done by giving an example.

For each $\alpha\in (1, +\infty)$, we put
\begin{equation*}
	q_n(\alpha)=1-\exp\left (-{n}^\alpha \right )
\end{equation*}
and $\omega (\alpha)=(q_n (\alpha))_{n=1}^{\infty}\in \Omega$.
Then, $\omega (\alpha)$ is increasing and $\lim_{n\to \infty}q_n(\alpha)=1$.
Furthermore, for $1\leq \alpha' <\alpha<+\infty$,
\begin{equation*}
	\frac{\log (1-q_n(\alpha))}{\log (1-q_n(\alpha'))}=\left ({n}\right )^{\alpha-\alpha'}\to \infty \quad (n\to \infty).
\end{equation*}
Hence, $\mathcal M(\omega(\alpha))\not=\mathcal M(\alpha')$ from Theorem \ref{Thm:condition} and we conclude that $\{\mathcal M(\alpha)\}_{\alpha\in (1, +\infty)}$ is uncountable.

\end{document}